\documentclass{amsart}

\newtheorem{theorem}{Theorem}
\newtheorem{lemma}{Lemma}

\newtheorem{proposition}{Proposition}

\theoremstyle{example}

\theoremstyle{definition}

\theoremstyle{remark}
\newtheorem{remark}{Remark}
\theoremstyle{remark}

\DeclareMathOperator{\am}{am}

\begin{document}

\title{The angle along a curve and range-kernel complementarity}
\author{Dimosthenis Drivaliaris}
\address{Department of Financial and Management Engineering\\
University of the Aegean
41, Kountouriotou Str.\\
82100 Chios\\
Greece}
\email{d.drivaliaris@fme.aegean.gr}
\author{Nikos Yannakakis}
\address{Department of Mathematics\\
National Technical University of Athens\\
Iroon Polytexneiou 9\\
15780 Zografou\\
Greece}
\email{nyian@math.ntua.gr}
\commby{}

\begin{abstract}
In this paper, we define the angle of a bounded linear operator $A$ along an unbounded path emanating from the origin and use it to characterize range-kernel complementarity. In particular we show that if $0$ faces the unbounded component of the resolvent set, then $X=R(A)\oplus N(A)$ if and only if $R(A)$ is closed and some angle of $A$ is less than $\pi$.
\end{abstract}

\maketitle
\section{Introduction}
Let $X$ be a Banach space and $A:X\rightarrow X$ be a bounded linear operator. We will denote the range of $A$ by $R(A)$ and the kernel of $A$ by $N(A)$.

The cosine of $A$, with respect to a semi-inner product $[\cdot, \cdot]$ compatible with the norm of $X$, is defined by
\begin{equation}
\nonumber
\cos A=\inf\left\lbrace\frac{Re[Ax, x]}{\|Ax\|\,\|x\|}\,:\,x\notin N(A)\right\rbrace\,.
\end{equation}
Using this one can define the angle
$\phi(A)$ of the linear operator $A$ by
\begin{equation}
\nonumber
\phi(A)=\arccos(\cos A)\,.
\end{equation}
The angle $\phi(A)$ of $A$ has an obvious geometric interpretation; it measures the maximum (real) turning effect of $A$. This concept was introduced independently by K. Gustafson in \cite{gus1} and by M. Krein in \cite{krein}.

A moment's thought reveals that $\phi(A)$  is just one of the many existng angles of the operator $A$. Indeed if $\theta\in [0,2\pi]$ is any angle, then $\phi(e^{i\theta}A)$ measures the maximum  turning effect of $A$ along the ray emanating from the origin
\[\{0\}\cup\left\{z\in\mathbb C: \arg z=\theta\right\}\,.\]
Related to the above  is  the so called ``amplitude angle'' of $A$, also  introduced by M. Krein in \cite{krein}, defined by
\[\am(A)=\min\left\{\phi(e^{i\theta}A):\,\theta\in [0,2\pi]\right\}\,.\]
The``amplitude angle'' of $A$ compares the angles along all possible ``directions'' and provides us with the smallest one.

Range-kernel complementarity, i.e. the decomposition
\begin{equation}
\label{compl}
X=R(A)\oplus N(A)\,,
\end{equation}
stands right next to the invertibility of $A$, since if (\ref{compl}) holds then $A$ is of the form ``invertible $\oplus\, 0$''. 

In finite dimensions (\ref{compl}) is equivalent to $R(A)\cap N(A)=\left\{0\right\}$
which in turn is equivalent to the ascent (the length of the null-chain) of  $A$ being less than or equal to one. 
In infinite dimensions things are significantly different as $R(A)\cap N(A)=\left\{0\right\}$  is no longer sufficient and one needs the additional assumption that  $R(A)=R(A^2)$. Note that the latter is equivalent to the descent (the lenght of the range chain) of  $A$ being less than or equal to one. 

In \cite[Theorem 3.4 and Remark 3.5 (vi)]{drivyann} it was shown that if $\am(A)<\pi$ and both $R(A)$ and $R(A)+N(A)$ are closed then (\ref{compl}) holds. The converse of this result is not true in general and the reason is simple:
as is mentioned by M. Krein in \cite{krein} $\am(A)<\pi$ implies that, for some $\theta\in [0,2\pi]$, the spectrum of $e^{i\theta}A$  is  contained in the sector 
\[S=\{0\}\cup\left\{z\in\mathbb C: |\arg z|\leq \phi(e^{i\theta}A)\right\}\,.\]
Hence  $\am(A)<\pi$ would imply that $\sigma(A)$ is contained in the sector $e^{-i\theta}S$  which is obviously not the case for an arbitrary operator in a infinite dimensional Banach space.

Since the inclusion of $\sigma(A)$ in some sector of the complex plane implies that there exists a ray emanating from the origin that does not intersect $\sigma(A)$,  it seems plausible to ask whether the converse is true if instead of rays one allows curves. As we will see the answer to this question turns out to be affirmative. To arrive at this conclusion we study operators for which there exists a curve emanating from the origin that does not intersect their spectrum, by defining their angle  along such curves. Having this in hand we  show that in this case (\ref{compl}) holds  if and only if $R(A)$ is closed and some such angle of $A$ is less than $\pi$.


\section{Preliminaries}
The ascent $\alpha(A)$ of $A$ is the smallest positive integer $k$ for which the null-chain  of $A$  terminates i.e. $N(A^k)=N(A^{k+1})$. If no such integer exists, then $\alpha(A)=\infty$. The descent $\delta(A)$ of $A$ is the smallest positive integer $k$ for which the range chain terminates i.e. $R(A^k)=R(A^{k+1})$. Again if no such integer exists, then $\delta(A)=\infty$. It is well known, see for example \cite[Proposition 38.4]{heuser}, that $ X=R(A)\oplus  N(A)$ if and only if both $ \alpha(A)$ and $\delta(A)$ are $\leq 1$.

If  by $d(x,N(A))$ we denote the distance of $x$ from $N(A)$, then there exists $M>0$ such that

\begin{equation}
\label{distance}
\|Ax\|\geq M d(x,N(A))\,,\text{ for all }x\in X\,.
\end{equation}

By $\partial \sigma(A)$ we denote as usual the boundary of the spectrum of the operator $A$ and by $\sigma_{app}(A)$ its approximate point spectrum. It is well-known that  $\partial \sigma(A)\subseteq\sigma_{app}(A)$ .
\section{The angle along a curve}
\subsection{Some motivation}

Let $H$ be a complex Hilbert space. If $x, y\in H$ are non-zero vectors then the number
\[\theta(x,y)=\arccos\frac{Re\langle x,y\rangle}{\|x\|\|y\|}\]
is usually called the angle between $x$ and $y$ (see for example \cite{krein}).

Recall that
\begin{equation}
\nonumber
\inf_{\lambda\leq 0}\frac{\|x-\lambda y\|}{\|x\|}=1
\end{equation}
is equivalent to $\theta(x,y)$ being acute. Moreover it can be easily seen that if this is not the case then
\begin{equation}
\label{identity}
\left(\inf_{\lambda\leq 0}\frac{\|x- \lambda y\|}{\|x\|}\right)^2+\left(\frac{Re\langle x,y\rangle}{\|x\|\|y\|}\right)^2=1\,.
\end{equation}
Hence  taking
\[\inf_{\lambda\leq 0}\frac{\|x-\lambda y\|}{\|x\|}\]
as our starting point we can say that  $\theta(x,y)$ is acute if the  infimum is one and is equal to
\[\pi-\arcsin \inf_{\lambda\leq 0}\frac{\|x-\lambda y\|}{\|x\|}\]
otherwise. Note that taking into account that $H$ is a complex Hilbert space a more appropriate name for $\theta(x,y)$   is ``the angle between $x$ and $y$ along the negative axis''.

\subsection{The angle along a curve}

For the rest of this paper $C$  will be an unbounded curve in the complex plane emanating from the origin. If $X$ is a complex Banach space and  $x, y\in X$, with $x\neq 0$  then we define $s_C(x,y)$ by
\[s_C(x,y)=\inf_{\lambda\in C}\frac{\|x-\lambda y\|}{\|x\|}\,.\]

Considering the above discussion we say that the angle $\theta_C(x,y)$ of $x$, $y$ along $C$ is acute  if $s_C(x,y)=1$ and is equal to $\pi-\arcsin s_C(x,y)$ otherwise.

The sine of the operator $A$ along $C$ is then
\begin{equation}
\nonumber
\sin_C A=\inf_{x\notin N(A)}s_C(Ax,x)\,.
\end{equation}
If $C$ is the negative axis instead of $s_C(Ax,x)$ and $\sin_C A$ we will just write $s(Ax,x)$ and $\sin A$.


Recall  that $A$ is called accretive if
\begin{equation}
\nonumber
\inf_{\lambda\leq 0}\frac{\|Ax-\lambda x\|}{\|Ax\|}=1\,,\text{ for all }x\notin N(A)
\end{equation}
which  in a Hilbert space $H$ is equivalent to $Re\langle Ax, x \rangle\geq 0$, for all $x\in H$.

We have the following.
\begin{proposition}
If $A\in B(H)$ is not accretive, then
\[\sin^2 A+\cos^2 A=1\,.\]
\end{proposition}
\begin{proof}
If $A$ is not accretive then the set
\[M=\left\{x\notin N(A):\,Re\langle Ax,x\rangle<0\right\}=\left\{x\notin N(A):\,s(Ax,x)<1\right\}\]
is nonempty and  obviously
\[\sin A=\inf_{x\in M}s(Ax,x)\,.\]
If
\[c(Ax,x)=\frac{Re\langle Ax, x\rangle}{\|Ax\|\,\|x\|}\,,\text{ for }x\notin N(A)\,.\]
we have that
\begin{eqnarray*}
\sin^2 A&=&\inf_{x\in M}s(Ax,x)\\
&=&\inf_{x\in M}[1-c(Ax,x)^2]\\
&=&1-\sup_{x\in M}c(Ax,x)^2\\
&=&1-(\inf_{x\in M}c(Ax,x))^2
\end{eqnarray*}
where the second equality is due to (\ref{identity}) and the last holds because $c(Ax,x)<0$, for all $x\in M$. Hence
\[\sin^2 A+\cos^2 A=1\,.\]
\end{proof}
\begin{remark}
In \cite[Theorem 3.2-1]{gus2} it is shown that if $A$ is strongly accretive then another sine satisfying the basic trigonometric identity may be defined by 
$$
\min_{\varepsilon >0}\|\varepsilon A-I\|\,.
$$
\end{remark}

The angle of $A$ along  $C$ is defined to be
\[\phi_C(A)=\arcsin\sin_C A\,.\]
As long as A is not accretive along $C$ (i.e.  the infimum is less than one), the angle $\phi_C(A)$ may be given a geometric interpretation; it measures the maximum turning effect of $A$ along $C$.
\section{Results}
Our starting point is the  following  Lemma.
\begin{lemma}
\label{lemma}
Let $A:X\rightarrow X$ be a bounded linear operator. If $\phi_C(A)<\pi$ for some curve $C$, then $R(A)\cap N(A)=\left\{0\right\}$ and $\overline{R(A)}+N(A)$ is closed.
\end{lemma}
\begin{proof}
 If $\phi_C(A)<\pi$ then there exists $0<c\leq 1$ such that
\begin{equation}
\nonumber
\|Ax+\lambda x\|\geq c\|Ax\|\,,\text{ for all }x\in X \text{ and all }\lambda\in C\,.
\end{equation}
 Let $x\in X$ and $y\in N(A)$. Since $C$ emanates from the origin there exists a sequence $(\lambda_n)$  of points in $C$ with $\lambda_n\rightarrow 0$, as $n\rightarrow\infty$ and  the above inequality  implies that
\[\|A(x+\frac{y}{\lambda_n})+\lambda_n(x-\frac{y}{\lambda_n})\|\geq c\|A(x+\frac{y}{\lambda_n})\|\,,\]
hence
\[\|Ax+\lambda_n x+y\|\geq c\|Ax\|\,,\text{ for all } n\in\mathbb N\,.\]
Letting $n\rightarrow+\infty$ we get
\begin{equation}
\label{closed}
\|Ax+y\|\geq c\|Ax\|\,,
\end{equation}
for all $x\in X$, $y\in N(A)$ and thus $\overline{R(A)}+ N(A)$ is closed. Moreover if $Ax\in N(A)$, then (\ref{closed}) implies that
\[\|Ax-Ax\|\geq c\|Ax\|\]
and hence $Ax=0$.
\end{proof}
In what follows by $D_\infty$ we denote the unbounded component of the resolvent set of $A$. We will use the following result which is a corollary of the so called ``filling the hole'' theorem.
\begin{proposition}[\cite{radjavi}, Theorem 0.8]
\label{unbounded}
If $M$ is a closed invariant subspace of $A$, then
\[\sigma(A|_M)\cap D_\infty=\emptyset\,.\]
\end{proposition}

Note that the existence of the unbounded curve emanating from the origin is equivalent to $0\in  \overline{D}_\infty$. Our main result is the following. 
\begin{theorem}
\label{main}
Let $A:X\rightarrow X$ be a bounded linear operator and assume that $0\in \overline{D}_\infty$.
Then
$$X=R(A)\oplus N(A)$$
if and only if $R(A)$ is closed and $\phi_C(A)<\pi$, for some $C$.
\end{theorem}
\begin{proof}
If $X=N(A)\oplus R(A)$ then there exists $\delta>0$ such that
\begin{equation}
\label{closed_1}
\|Ax+y\|\geq\delta \|Ax\|\,,\text{ for all } x\in X, y\in N(A)\,.
\end{equation}
By Proposition \ref{unbounded} we have that
\[\sigma(A|_{R(A)})\cap D_\infty=\emptyset\,.\]
Hence $D_\infty$ lies in the resolvent of $A|_{R(A)}$ and thus there exists an unbounded path $C$ emanating from the origin with
\[C\subseteq\rho(A|_{R(A)})\,.\]
Since $A|_{R(A)}$ is invertible we claim that there exists $k>0$ such that
\begin{equation}
\label{uniform}
\|Ax-\lambda x\|\geq k\|x\|\,,\text{ for all } x\in R(A), \lambda\in C\,.
\end{equation}
Assume the contrary;  i.e., that there exists a sequence $(\lambda_n)$ in $C$ and a sequence $(x_n)$ in $R(A)$, with $\|x_n\|=1$, such that $\|Ax_n-\lambda_n x_n\|\rightarrow 0$, as $n\rightarrow\infty$. Then since $(\lambda_n)$ is bounded it has a subsequence, which for simplicity we denote again by $(\lambda_n)$, that converges to some $\lambda_0\in C$. But then
\[\|Ax_n-\lambda_0x_n\|\leq \|Ax_n-\lambda_nx_n\|+|\lambda_n-\lambda_0|\]
and hence $\|Ax_n-\lambda_0x_n\|\rightarrow 0$, as $n\rightarrow\infty$, which is a contradiction since $\lambda_0\in\rho(A|_{R(A)})$.
We will show that $\phi_C(A)<\pi$.

To this end let $x\in X$. By hypothesis $x=z+y$, with $z\in R(A)$ and $y\in N(A)$. Using (\ref{closed_1}) and (\ref{uniform}) we have that
\[\|Ax+\lambda x\|=\|Az+\lambda z+\lambda y\|\geq\delta\|Az+\lambda z\|\geq \delta k\|z\|\geq \frac{\delta k}{\|A\|}\|Az\|=c\|Ax\|\,.\]
Hence $s_C(Ax,x)\geq c$, for all $x\notin N(A)$ and so $\phi_C(A)<\pi$.

Conversly, assume that  $\phi_C(A)<\pi$, for some $C$. Then by  Lemma \ref{lemma} we have that $R(A)\cap N(A)=\left\{0\right\}$ and hence the ascent $\alpha(A)$ of $A$ is lesser or equal to 1. To conclude the proof we have to show that the descent $\delta(A)$ is finite. In particular we will show that $R(A^2)=R(A)$. 

By  Lemma \ref{lemma}, using the fact that $R(A)$ is closed, we have that $R(A)+N(A)$ is a closed subspace of $X$ which implies that $R(A^2)$ is also closed. Hence since $R(A)\cap N(A)=\left\{0\right\}$ by (\ref{distance}) we have that
\begin{equation}
\label{eq1}
\|Ax\|\geq\|x\|\,,\text{ for all }x\in R(A)\,.
\end{equation}
If $0\in\sigma(A|_{R(A)})$ and since by Proposition \ref{unbounded} we have that
\[\sigma(A|_{R(A)})\cap D_\infty=\emptyset\]
we get that $0\in\partial\sigma(A|_{R(A)})\subseteq\sigma_{app}(A|_{R(A)})$ which contardicts (\ref{eq1}). So $0\notin\sigma(A|_{R(A)})$ and the proof is complete.
\end{proof}

Recall that a hole in a compact subset of the complex plane is a bounded connected component of its complement. Using the above we can  show that  if we have range-kernel complementarity then $0$ faces the unbounded component of the resolvent set if and only if some angle of $A$ is less than $\pi$.
\begin{proposition}
\label{holes}
Let $A\in B(X)$ and assume that $X=N(A)\oplus R(A)$. Then $0\in \overline{D}_\infty$ if and only if $\phi_C(A)<\pi$, for some $C$.
\end{proposition}
\begin{proof}
If $0\in\overline{D}_\infty$ the result follows from Theorem \ref{main}. Conversly assume that $0$ is inside a hole in $\sigma(A)$.  We will show that all angles of $A$ are equal to $\pi$. To see this note that in this case every unbounded path $C$ emanating from the origin intersects $\partial\sigma(A)$ and in patricular $\sigma_{app}(A)$ and hence if  $\lambda_0$ is this point of intersection there exists a sequence $(x_n)$ in $X$ such that
\[\|Ax_n+\lambda_0 x_n\|\rightarrow 0\,,\text{ as } n\rightarrow\infty\,.\]
Since $X=N(A)\oplus R(A)$ each $x_n$ may be written as $x_n=y_n+z_n$, with $y_n\in N(A)$ and $z_n\in R(A)$. But then and since the sum $N(A)\oplus R(A)$ is closed we have that
\[\|Ax_n+\lambda_0 x_n\|=\|Az_n+\lambda_0 y_n+ \lambda_0 z_n\|\geq c' \|Az_n+ \lambda_0 z_n\|\]
and hence
\begin{equation}
\label{zero}
\|Az_n+\lambda_0 z_n\|\rightarrow 0\,,\text{ as } n\rightarrow\infty\,.
\end{equation}
On the other hand if $\phi_C(A)<\pi$, then by $X=N(A)\oplus R(A)$, the range of $A|_{R(A)}$ is closed and $N(A)\cap R(A)=\left\{0\right\}$. Hence
\[\|Ax+\lambda_0 x\|\geq c\|Ax\|\geq k\|x\|\,,\]
for all $x\in R(A)$. But this contradicts (\ref{zero}) and thus $\phi_C(A)=\pi$, for all $C$. So $\phi_C(A)<\pi$ implies that $0\in D_\infty$.
\end{proof}

\bibliographystyle{amsplain}


\end{document}